\pdfoutput=1
\RequirePackage{ifpdf}
\ifpdf % We are running pdfTeX in pdf mode
\documentclass[pdftex]{sigma}
\else
\documentclass{sigma}
\fi

\usepackage{amscd}
\usepackage[all]{xy}

\numberwithin{equation}{section}

\newtheorem{Theorem}{Theorem}[section]

\newtheorem{Lemma}[Theorem]{Lemma}
\newtheorem{Proposition}[Theorem]{Proposition}

 { \theoremstyle{definition}
\newtheorem{Definition}[Theorem]{Definition}

\newtheorem{Example}[Theorem]{Example}
\newtheorem{Remark}[Theorem]{Remark}

 }

\newcommand{\NN}{{\mathbb{N}}} %natural numbers
\newcommand{\QQ}{{\mathbb{Q}}} %rational numbers
\newcommand{\RR}{{\mathbb{R}}} %real numbers or cartesian space
\newcommand{\TT}{{\mathbb{T}}} %torus
\newcommand{\ZZ}{{\mathbb{Z}}} %integers
\newcommand{\inter}{{\operatorname{int}}} %interior
\newcommand{\pr}{{\operatorname{pr}}} %projection
\newcommand{\g}{{\mathfrak{g}}} %generic Lie algebra
\newcommand{\SO}{{\operatorname{SO}}} %special orthogonal
\newcommand{\CIN}{{C^\infty}} %infinitely differentiable
\newcommand{\hook}{{\lrcorner\,}} %hook or interior multiplication
\def \red {{\text{red}}}
\def \reg {{\text{reg}}}

\newcommand{\toto}{{\rightrightarrows}}
\def \calD {{\mathcal{D}}}

\def \tU {{\widetilde{U}}}
\DeclareMathOperator {\image}{Image}

\newcommand{\hypref}[2]{{#2~\ref{#1}}}

\begin{document}

%\allowdisplaybreaks

\newcommand{\arXivNumber}{1408.1555}

\renewcommand{\PaperNumber}{026}

\FirstPageHeading

\ShortArticleName{Basic Forms and Orbit Spaces: a Dif\/feological Approach}

\ArticleName{Basic Forms and Orbit Spaces:\\ a Dif\/feological Approach}

\Author{Yael KARSHON~$^\dag$ and Jordan WATTS~$^\ddag$}

\AuthorNameForHeading{Y.~Karshon and J.~Watts}

\Address{$^\dag$~Department of Mathematics, University of Toronto,\\
\hphantom{$^\dag$}~40 St.\ George Street, Toronto Ontario M5S 2E4, Canada}
\EmailD{\href{mailto:karshon@math.toronto.edu}{karshon@math.toronto.edu}}
\URLaddressD{\url{http://www.math.toronto.edu/karshon/}}

\Address{$^\ddag$~Department of Mathematics, University of Colorado Boulder,\\
\hphantom{$^\ddag$}~Campus Box 395, Boulder, CO, 80309, USA}
\EmailD{\href{mailto:jordan.watts@colorado.edu}{jordan.watts@colorado.edu}}
\URLaddressD{\url{http://euclid.colorado.edu/~jowa8403}}

\ArticleDates{Received October 06, 2015, in f\/inal form February 16, 2016; Published online March 08, 2016}	

\Abstract{If a Lie group acts on a manifold freely and properly,
pulling back by the quotient map gives an isomorphism between the
dif\/ferential forms on the quotient manifold
and the basic dif\/ferential forms upstairs.
We show that this result remains true for actions that are not necessarily
free nor proper, as long as the identity component
acts properly, where on the quotient space we take dif\/ferential forms
in the dif\/feological sense.}


\Classification{58D19; 57R99}

\vspace{-2mm}

\section{Introduction}\label{sec:intro}

Let $M$ be a smooth manifold and $G$ a Lie group acting on $M$.
A \emph{basic differential form} on $M$
is a dif\/ferential form that is $G$-invariant and horizontal;
the latter means that
evaluating the form on any vector that is tangent to a $G$-orbit yields $0$.
Basic dif\/ferential forms constitute a~subcomplex of the de Rham complex.
If $G$ acts properly and with a constant orbit-type, then the quotient~$M/G$ is a manifold,
and, denoting the quotient map by $\pi \colon M \to M/G$,
the pullback by this map
gives an isomorphism of the de Rham complex on~$M/G$
with the complex of basic forms on~$M$.
Even if~$M/G$ is not a manifold, if $G$ acts properly, then
the cohomology of the complex of basic forms
is isomorphic to the singular cohomology of $M/G$ with real coef\/f\/icients;
this was shown by Koszul in 1953~\cite{koszul} for compact group actions
and by Palais in 1961~\cite{palais} for proper group actions.
In light of these facts, some authors def\/ine the de Rham complex on~$M/G$
to be the complex of basic forms on~$M$.

There is another, intrinsic, def\/inition of a dif\/ferential form on $M/G$,
which comes from viewing $M/G$ as a dif\/feological space (see Section~\ref{sec:diffeology}). This def\/inition agrees with the usual one when $M/G$ is a manifold. Dif\/ferential forms on dif\/feological spaces
admit exterior derivatives,
wedge products, and pullbacks under smooth maps.
Spaces of dif\/ferential forms are themselves dif\/feological spaces too.
With this notion, here is our main result:

\begin{Theorem}\label{t:main}\looseness=-1
Let $G$ be a Lie group acting on a manifold $M$.
Let $\pi\colon M\to M/G$ be the quotient map.
\begin{enumerate}\itemsep=0pt
\item[$(i)$] The pullback map $\pi^* \colon \Omega^*(M/G) \to \Omega^*(M)$
is one-to-one. Its image is contained in the space $\Omega^*_{\rm basic}(M)$
of basic forms. As a map to its image,
the map $\pi^*$ is an isomorphism of differential graded algebras
and a diffeological diffeomorphism.
\item[$(ii)$] If the restriction of the action to the identity component of~$G$
is proper, then the image of the pullback map is equal to the space
of basic forms:
\begin{gather*}\xymatrix{
\pi^* \colon \ \Omega^*(M/G) \ar[r]^(.55){\cong}
& \Omega^*_{\rm basic}(M)}.\end{gather*}
\end{enumerate}
\end{Theorem}

We prove Theorem~\ref{t:main} in Section~\ref{sec:surjects},
after the proof of Proposition~\ref{p:geomquot}.

\begin{Remark} \quad
\begin{enumerate}\itemsep=0pt
\item
Part~(i) of Theorem~\ref{t:main}, which follows from
the results of Sections~\ref{sec:diffeology} and~\ref{sec:injects},
is not dif\/f\/icult. The technical heart of Theorem~\ref{t:main} is Part (ii), which is proved in \hypref{p:geomquot}{Proposition}: if the identity component of the group acts properly, then every basic form on $M$ descends to a~dif\/feological form on~$M/G$.
This fact is non-trivial even when the group~$G$ is f\/inite.

\item
The quotient $M/G$ can be non-Hausdorf\/f. Nevertheless, even if its topology
is trivial, $M/G$ may have non-trivial dif\/ferential forms,
and its de Rham cohomology may be non-trivial.
See, for example, the irrational torus in~\hypref{rk:irrational torus}{Remark}.

\item
We expect
the conclusion of Part~(ii) of Theorem~\ref{t:main} to hold
under more general hypotheses.
In particular, our assumption that the identity component of $G$
act properly on $M$ is suf\/f\/icient but not necessary;
see, for instance, \hypref{x:irrational solenoid}{Example}.

\end{enumerate}
\end{Remark}

Dif\/feology was developed by Jean-Marie Souriau (see \cite{souriau})
around 1980, following earlier work of Kuo-Tsai Chen
(see, e.g., \cite{chen1,chen4}).
Our primary reference for this theory is the book~\cite{iglesias}
by Iglesias-Zemmour.
In many applications, dif\/feology can serve as a replacement
for the manifold structures
(modelled on locally convex topological vector spaces)
on spaces of smooth paths, functions, or dif\/ferential forms.
See, for example, \cite{BFW}.

The category of dif\/feological spaces is complete and co-complete
(see \cite{BH}); in particular, subsets and quotients
naturally inherit dif\/feological structures.
It is also Cartesian closed, where we equip spaces of smooth maps
with their natural functional dif\/feology.

It is also common to consider $M/G$ as a \emph{$($Sikorski$)$ differential space},
by equipping it with the set of those real valued functions
whose pullback to $M$ is smooth.
See Hochschild~\cite{hochschild}, Bredon~\cite{bredon},
G.~Schwarz~\cite{schwarz}, and Cushman and \'Sniatycki~\cite{CuSn}.
This structure is determined by the dif\/feology on~$M/G$ but is weaker.
For example, the quotients~$\RR^n/\SO(n)$
for dif\/ferent positive integers $n$ are isomorphic as dif\/ferential spaces
but not as dif\/feological spaces.
(See Exercise~50 of Iglesias~\cite{iglesias}
with solution at the end of the book.)
There are several inequivalent notions of ``dif\/ferential form''
on dif\/ferential spaces
(see \'Sniatycki~\cite{sniatycki} and Watts~\cite{watts-masters});
we do not know of an analogue of \hypref{t:main}{Theorem}
for any of these notions.

Turning to higher category theory,
we can also consider the stack quotient $[M/G]$.
It is a dif\/ferentiable stack over the site of manifolds,
and it is represented by the action groupoid $G\times M\toto M$.
One can def\/ine a dif\/ferential $k$-form on the stack $[M/G]$
as a map of stacks from $[M/G]$
to the stack of dif\/ferential forms $\Omega^k$.
In~\cite{WW},
Watts and Wolbert def\/ine a functor $\mathbf{Coarse}$ from stacks
to dif\/feological spaces for
which $\mathbf{Coarse}([M/G])$ is equal to $M/G$ equipped with the
quotient dif\/feology. In this language,
\hypref{t:main}{Theorem} gives an isomorphism
\begin{gather*} \Omega^k([M/G])\cong \Omega^k(\mathbf{Coarse}([M/G])), \end{gather*}
when the identity component of $G$ acts properly.
We note that
a dif\/feological space can also be viewed as a stack,
and applying $\mathbf{Coarse}$ recovers the original dif\/feological space.
We also note that
the quotient stack $[M/G]$ often contains more information than
the quotient dif\/feology;
for example, if~$M$ is a point, the quotient dif\/feological space is a point,
but the stack $[M/G]$ determines the group~$G$.
Finally, we note that
Karshon and Zoghi~\cite{zoghi} give suf\/f\/icient conditions
for a Lie groupoid to be determined up to Morita equivalence
by its underlying dif\/feological space.

In the special case that the Lie group $G$ is compact,
the main results of this paper appeared in the Ph.D.\ Thesis of the
second author~\cite{watts}, supervised by the f\/irst author.
A generalisation to proper Lie groupoids, which relies on these results and on
a deep theorem of Crainic and Struchiner~\cite{CrSt}
showing that all proper Lie groupoids are linearisable,
was worked out by Watts in~\cite{watts-groupoids}.

The paper is structured as follows.
For the convenience of the reader,
Section~\ref{sec:diffeology} contains background on dif\/feology, dif\/ferential forms on dif\/feological spaces, and Lie group actions in connection to dif\/feology.
Section~\ref{sec:injects} contains a proof that the pullback map
from the space of dif\/ferential forms on the orbit space $M/G$
to the space of dif\/ferential forms on the manifold $M$
is an injection into the space of basic forms, and is a dif\/feomorphism onto its image.
Section~\ref{sec:q in stages} contains a technical lemma: ``quotient
in stages''.
Section~\ref{sec:surjects} is the technical heart of the paper;
it contains a proof that, if the identity component of the group acts properly,
then the pullback map surjects onto the space of basic forms.

Our appendices contain two applications
of the special case of \hypref{t:main}{Theorem}
when the group~$G$ is f\/inite.
In Appendix~\ref{sec:orbifolds} we show that, on an orbifold,
the notion of a dif\/feological dif\/ferential form
agrees with the usual notion of a dif\/ferential form on the orbifold.
In Appendix~\ref{sec:sjamaar} we show that, on a regular symplectic quotient
(which is also an orbifold),
the notion of a dif\/feological dif\/ferential form also agrees
with Sjamaar's notion of a dif\/ferential form on the symplectic quotient.
The case of non-regular symplectic quotients is open.

\section{Background on dif\/feological spaces}\label{sec:diffeology}

In this section we review the basics of dif\/feology, dif\/feological dif\/ferential forms, and Lie group actions in the context of dif\/feology. For more details (e.g., the quotient dif\/feology is in fact a~dif\/feology), see Iglesias-Zemmour~\cite{iglesias}.

\subsection*{The basics of dif\/feology}

This subsection contains a review of the basics of dif\/feology; in particular, the def\/inition of a~dif\/feology and dif\/feologically smooth maps, as well as various constructions in the dif\/feological category.

\begin{Definition}[dif\/feology]
Let $X$ be a set. A \emph{parametrisation} on $X$ is a
function $p \colon U \to X$ where $U$ is an open subset of $\RR^n$
for some $n$.
A \emph{diffeology} $\mathcal{D}$ on $X$ is a set of parametrisations
that satisf\/ies the following three conditions.

\begin{enumerate}\itemsep=0pt
\item (\emph{Covering})
For every point $x \in X$ and every non-negative integer $n \in \NN$,
the constant function $p \colon \RR^n \to \{ x \} \subseteq X$
is in $\mathcal{D}$.
\item (\emph{Locality})
Let $p \colon U \to X$ be a parametrisation such that
for every point in $U$ there exists an open neighbourhood $V$ in $U$
such that $p|_V \in \mathcal{D}$. Then $p\in\mathcal{D}$.
\item (\emph{Smooth compatibility})
Let $(p \colon U\to X) \in \mathcal{D}$.
Then for every $n \in \NN$, every open subset $V \subseteq \RR^n$,
and every smooth map $F \colon V \to U$, we have $p \circ F \in \mathcal{D}$.
\end{enumerate}

A set $X$ equipped with a dif\/feology $\mathcal{D}$ is called
a \emph{diffeological space} and is denoted by $(X,\mathcal{D})$.
When the dif\/feology is understood, we drop the symbol $\mathcal{D}$.
The elements of $\calD$ are called \emph{plots}.
\end{Definition}

\begin{Example}[standard dif\/feology on a manifold]
Let $M$ be a manifold. The \emph{standard diffeology} on $M$ is the
set of all smooth maps to $M$ from open subsets of $\RR^n$
for all $n \in \NN$.
\end{Example}

\begin{Definition}[dif\/feologically smooth maps]
Let $X$ and $Y$ be two dif\/feological spaces,
and let $F \colon X\to Y$ be a map. We say that $F$ is
\emph{(diffeologically) smooth} if for any plot $p \colon U \to X$ of $X$
the composition $ F \circ p \colon U \to Y$ is a plot of $Y$.
Denote by $\CIN(X,Y)$ the set of all smooth maps
from $X$ to $Y$.
Denote by $\CIN(X)$ the set of all smooth maps from $X$ to~$\RR$,
where $\RR$ is equipped with its standard dif\/feology.
\end{Definition}

\begin{Remark}[plots]
A parametrisation is dif\/feologically smooth if and only if it is a plot.
\end{Remark}

\begin{Remark}[smooth maps between manifolds]
A map between two manifolds is dif\/feologically smooth
if and only if it is smooth in the usual sense.
In particular, if $M$ is a manifold then $\CIN(M)$ is the usual set
of smooth real valued functions.
\end{Remark}

\begin{Remark}
Dif\/feological spaces, along with dif\/feologically smooth maps,
form a category. It is shown in \cite[Theorem~3.2]{BH} that this category
is a complete and cocomplete quasi-topos.
In particular, it is closed under passing to arbitrary quotients,
subsets, function spaces, products, and coproducts.
\end{Remark}

\begin{Definition}[quotient dif\/feology]
Let $X$ be a dif\/feological space,
and let $\sim$ be an equivalence relation on $X$.
Let $Y = X/\!\!\sim$ be the quotient set,
and let $\pi \colon X\to Y$ be the quotient map.
We def\/ine the \emph{quotient diffeology} on $Y$
to be the dif\/feology for which the plots are those maps $p \colon U\to Y$
such that for every point in $U$
there exist an open neighbourhood $V \subseteq U$
and a plot $q \colon V\to X$ such that $p|_V = \pi \circ q$.
\end{Definition}

\begin{Remark}[quotient map]
Let $X$ be a dif\/feological space and $\sim$ an equivalence relation on~$X$.
Then the quotient map $\pi \colon X \to X / \!\!\sim$ is smooth.
\end{Remark}

A special case that is important to us is the quotient
of a manifold by the action of a Lie group.

\begin{Definition}[subset dif\/feology]
Let $X$ be a dif\/feological space, and let~$Y$ be a subset
of~$X$. The \emph{subset diffeology} on~$Y$ consists
of those maps to~$Y$ whose composition with the inclusion map $Y \to X$
are plots of~$X$.
\end{Definition}

\begin{Definition}[product dif\/feology]
Let $X$ and $Y$ be two dif\/feological spaces.
The \emph{product diffeology} on the set $X \times Y$ is def\/ined as follows.
Let $\operatorname{pr}_X \colon X \times Y \to X$
and $\operatorname{pr}_Y \colon X \times Y \to Y$ be the natural projections.
A parametrisation $p \colon U \to X \times Y$ is a plot
if $\operatorname{pr}_X \circ p$ and $\operatorname{pr}_Y \circ p$
are plots of $X$ and $Y$, respectively.
\end{Definition}

\begin{Definition}[standard functional dif\/feology on maps]
\label{d:functional}
Let $Y$ and $Z$ be dif\/feological spaces.
The \emph{standard functional diffeology} on $\CIN(Y,Z)$ is def\/ined as follows.
A parametrisation $p \colon U \to \CIN(Y,Z)$ is a plot if the map
 \begin{gather*} U \times Y\to Z \qquad \text{given by} \quad (u,y)\mapsto p(u)(y) \end{gather*}
is smooth.
\end{Definition}

\subsection*{Dif\/feological dif\/ferential forms}

In this subsection we review dif\/ferential forms on dif\/feological spaces, as well as introduce Proposition~\ref{p:pullbackimgchar} which is a~simple criterion crucial to the proof of Theorem~\ref{t:main}.

\begin{Definition}[dif\/ferential forms]
\label{d:diff forms}
Let $(X,\mathcal{D})$ be a dif\/feological space.
A \emph{$($diffeological$)$ differential $k$-form} $\alpha$ on $X$
is an assignment to each plot $( p \colon U \to X) \in \mathcal{D}$
a dif\/ferential $k$-form $\alpha(p) \in \Omega^k(U)$
satisfying the following \emph{smooth compatibility} condition:
for every open subset $V$ of a~Euclidean space
and every smooth map $F \colon V \to U$,
\begin{gather*} \alpha(p\circ F) = F^* (\alpha(p)) .\end{gather*}
Denote the set of dif\/ferential $k$-forms on $X$ by $\Omega^k(X)$.
\end{Definition}

\begin{Definition}[pullback map]
\label{d:pullback}
Let $X$ and $Y$ be dif\/feological spaces, and let $F \colon X\to Y$
be a dif\/feologically smooth map.
Let $\alpha$ be a dif\/ferential $k$-form on $Y$.
Def\/ine the \emph{pullback} $F^*\alpha$ to be the $k$-form on $X$
that satisf\/ies the following condition:
for every plot $p \colon U \to X$,
\begin{gather*} F^* \alpha(p) = \alpha(F\circ p) .\end{gather*}
\end{Definition}

\begin{Example} \label{e:forms-on-mflds}
Let $\alpha$ be a dif\/ferential form on a manifold $M$.
Then $(p \colon U \to M) \mapsto p^* \alpha$
def\/ines a dif\/feological dif\/ferential form on $M$.
In this way, we get an identif\/ication of the ordinary dif\/ferential forms
on $M$ with the dif\/feological dif\/ferential forms on $M$.

Let $X$ be a dif\/feological space, $\alpha$ a dif\/ferential form on $X$,
and $p \colon U \to X$ a plot.
The above identif\/ication of ordinary dif\/ferential forms on $U$
with dif\/feological dif\/ferential forms on $U$
gives $\alpha(p) = p^* \alpha$.
Henceforth, we may write $p^*\alpha$ instead of $\alpha(p)$.

\end{Example}

\begin{Example} \label{e:Omega0}
The space $\Omega^0(X)$ of dif\/feological $0$-forms
is identif\/ied with the space $\CIN(X)$
of smooth real valued functions,
by identifying the function $f$
with the $0$-form $(p \colon U \to X) \mapsto f \circ p$.
See \cite[Section~6.31]{iglesias}.
With this identif\/ication, the pullback of $0$-forms by a smooth map~$F$
becomes the precomposition of smooth real-valued functions by~$F$.
\end{Example}

\begin{Remark}[pullback is linear]\label{r:linear}
The space of dif\/ferential forms on a dif\/feological space is naturally
a linear vector space:
for $\alpha,\beta \in \Omega^k(X)$ and $a,b \in \RR$,
we def\/ine $a\alpha+b\beta \colon (p \colon U \to X) \mapsto
a \alpha(p) + b \beta(p)$.
If $F \colon X \to Y$ is a smooth map of dif\/feological spaces,
then the pullback map $F^*	 \colon \Omega^k(Y) \to \Omega^k(X)$ is linear.
\end{Remark}

\begin{Remark}[wedge product and exterior derivative]
\label{alg iso}
Let $X$ be a dif\/feological space.
Def\/ine the \emph{wedge product}
of $ \alpha \in \Omega^k(X) $ and $ \beta \in \Omega^l(X) $
to be the $(k+l)$-form
$\alpha \wedge \beta \colon (p \colon U \to X)
 \mapsto p^* \alpha \wedge p^* \beta$.
Def\/ine the \emph{exterior derivative} of $\alpha$
to be the $k+1$ form
$d\alpha \colon (p \colon U \to X) \mapsto d (p^*\alpha)$.
Then $\Omega^*(X)=\bigoplus_{k=0}^\infty\Omega^k(X)$
is a dif\/ferential graded algebra.
In particular, $\Omega^*(X)$ is an exterior algebra,
and $(\Omega^*(X),d)$ is a complex.
If $F \colon X \to Y$ is a smooth map of dif\/feological spaces,
then the pullback map $F^* \colon \Omega^*(Y) \to \Omega^*(X)$
is a morphism of dif\/ferential graded algebras;
in particular,
it intertwines the wedge products and the exterior derivatives.
\end{Remark}

\begin{Definition}[standard functional dif\/feology on forms]
\label{d:D on forms}
Let $(X,\calD)$ be a dif\/feological space.
The \emph{standard functional diffeology} on $\Omega^k(X)$
is def\/ined as follows.
A parametrisation $p \colon U \to \Omega^k(X)$ is a plot
if for every plot $(q \colon V \to X) \in \mathcal{D}$,
where $V$ is open in $\RR^n$, the map $U \times V \to \bigwedge^k \RR^n$
sending $(u,v)$ to $q^*(p(u))|_v$ is smooth.
See \cite[Section~6.29]{iglesias} for a proof that this is indeed a~dif\/feology.
\end{Definition}

\begin{Remark}\label{r:formfacts}
Let $X$ be a dif\/feological space.
We have the following facts.
\begin{enumerate}\itemsep=0pt

\item %\label{f:Omega0}
Under the identif\/ication of \hypref{e:Omega0}{Example}
of the space of dif\/feological $0$-forms
with the space of smooth real valued functions,
Def\/initions~\ref{d:functional} and~\ref{d:D on forms}
of the standard functional dif\/feology on these spaces agree.

\item \label{f:pullback is sm}
If $F \colon X \to Y$ is a smooth map to another dif\/feological space,
then the pullback map $F^* \colon \Omega^k(Y) \to \Omega^k(X)$
is smooth with respect to the standard functional dif\/feologies
on the sets of dif\/ferential forms.
See \cite[Section~6.32]{iglesias}.

\item
The exterior derivative
$d \colon \Omega^k(X)\to\Omega^{k+1}(X)$ and the wedge product
$ \Omega^k (X) \times \Omega^l (X) \to \Omega^{k+l} (X)$ are smooth.
See Sections~6.34 and~6.35 of~\cite{iglesias}.
\end{enumerate}
\end{Remark}

\begin{Proposition}[pullbacks of quotient dif\/feological forms]\label{p:pullbackimgchar}
Let $G$ be a Lie group, acting on a manifold $M$,
and let $\pi \colon M \to M/G$ be the quotient map.
Then a differential form $\alpha$ on $M$ is in the image of $\pi^*$
if and only if, for every two plots $p_1 \colon U\to M$
and $p_2 \colon U \to M$ such that $\pi \circ p_1 = \pi \circ p_2$, we have
\begin{gather*} p_1^* \alpha = p_2^* \alpha. \end{gather*}
\end{Proposition}

\begin{proof}
This result is a special case of \cite[Section~6.38]{iglesias}.
\end{proof}

\subsection*{Group actions}

Here we highlight some facts about Lie group actions in connection to dif\/feology, as well as review the def\/inition of basic forms. Finally, we introduce another crucial ingredient to the proof of Theorem~\ref{t:main}, the slice theorem (Theorem~\ref{t:slice theorem}). Koszul \cite{koszul} proved the slice theorem for compact Lie group actions, and
Palais \cite{palais} proved it for proper Lie group actions.
The proof is also described in Theorem~2.3.3 of~\cite{DK}
and in Appendix~B of~\cite{GGK}.

\begin{Lemma} \label{facts}
Let a Lie group $G$ act on a manifold $M$.
Let $x$ be a point in $M$ and $H$ its stabiliser. Then
\begin{itemize}\itemsep=0pt
\item
There exists a unique manifold structure on the quotient $G/H$
such that the quotient map $G \to G/H$ is a submersion.
The standard diffeology on this manifold
agrees with the quotient diffeology induced from $G$.

\item
There exists a unique manifold structure on the orbit $G \cdot x$
such that the inclusion map $G \cdot x \to M$ is an immersion.
The standard diffeology on this manifold agrees with the subset diffeology
induced from~$M$.

\item
The orbit map $a \mapsto a \cdot x$ from $G$ to $M$
descends to a diffeomorphism from $G/H$ to $G \cdot x$.

\item
The tangent space $T_x(G \cdot x)$
is the space of vectors $\xi_M|_x$ for $\xi \in \g$,
where $\xi_M$ is the vector field on $M$ that is induced
by the Lie algebra element $\xi$.
This space is also the image of the differential at the identity
of the orbit map $a \mapsto a \cdot x$ from $G$ to $M$.
\end{itemize}
\end{Lemma}

\begin{proof}
See \cite[Section~2, Paragraph~1]{IK}.
\end{proof}

\begin{Definition}[basic forms] \label{d:basic form}
Let $G$ be a Lie group acting on a manifold $M$.
A dif\/ferential form $\alpha$ on $M$ is \emph{horizontal}
if for any $x \in M$ and $v \in T_x(G\cdot x)$ we have
\begin{gather*} v \hook \alpha = 0.\end{gather*}
(Recall that $v\hook\alpha=\alpha(v,\cdot,\dots,\cdot)$.)
A form that is both horizontal and $G$-invariant is called \emph{basic}.
When the $G$-action is understood,
we denote the set of basic $k$-forms on $M$ by $\Omega^k_{\rm basic}(M)$.
\end{Definition}

\begin{Remark} \label{r:wedge and d of basic}
The space of basic dif\/ferential forms on a $G$-manifold $M$
is closed under linear combinations, wedge products,
and exterior derivatives.
\end{Remark}

\begin{Remark}
Given a quotient map $\pi \colon X \to X'$ of dif\/feological spaces,
(more generally, given a so-called subduction,)
Iglesias-Zemmour~\cite[Section~6.38]{iglesias}
def\/ines a~``basic form'' to be a~dif\/ferential form on~$X$
that satisf\/ies the technical condition that appears in
Proposition~\ref{p:pullbackimgchar} above.
Our results show that, for smooth Lie group actions
where the identity component acts properly,
Iglesias-Zemmour's def\/inition agrees with the usual one.
\end{Remark}

Let $G$ be a Lie group, $H$ a closed subgroup, and $V$ a vector space
with a linear $H$-action. The equivariant vector bundle
\begin{gather*} G \times_H V \end{gather*}
over $G/H$ is obtained as the quotient of $G \times V$
by the anti-diagonal $H$-action $h \cdot (g,v) = (g h^{-1} , h \cdot v)$.
The $G$-action on $G\times_H V$ is $g\cdot[g',v]=[gg',v]$.

\begin{Theorem}[slice theorem]
\label{t:slice theorem}
Let $G$ be a Lie group acting properly on a manifold $M$. Fix $x\in M$.
Let $H$ be the stabiliser of $x$, and let $V = T_x M / T_x(G\cdot x)$
be the normal space to the orbit $G\cdot x$ at $x$,
equipped with the linear $H$-action that is induced
by the linear isotropy action of $H$ on $T_xM$.
Then there exist a $G$-invariant open neighbourhood $U$ of $x$
and a $G$-equivariant diffeomorphism $F\colon U\to G\times_H V$
that takes~$x$ to~$[1,0]$.
\end{Theorem}

\section{The pullback injects into basic forms}\label{sec:injects}

In this section we prove the easy part of Theorem~\ref{t:main}:
for a Lie group $G$ acting on a manifold $M$
with quotient map $\pi \colon M \to M/G$,
the pullback map $\pi^*$ is an injection
from the set of dif\/feological dif\/ferential forms on $M/G$
into the set of basic forms on $M$,
and $\pi^*$ is a dif\/feomorphism onto its image.

The following lemma is a special case of~\cite[Section~6.39]{iglesias}.

\begin{Lemma} \label{l:pullback is injection}
Let a Lie group $G$ act on a manifold $M$,
and let $\pi \colon M \to M/G$ be the quotient map.
Then the pullback map on forms,
$\pi^* \colon \Omega^k(M/G) \to \Omega^k(M)$, is an injection.
\end{Lemma}

\begin{proof}
More generally, let $X$ be a dif\/feological space, $\sim$ an equivalence
relation on $X$, and $\pi \colon X \to X/\!\!\sim$ the quotient map.
Then the pullback map on forms,
$\pi^* \colon \Omega^k(X/\!\!\sim) \to \Omega^k(X)$, is an injection.

By \hypref{r:linear}{Remark} it is enough to show that
the kernel of the pullback map
\begin{gather*}\pi^* \colon \ \Omega^k(X/\!\!\sim) \to \Omega^k(X)\end{gather*}
is trivial.
Let $\alpha \in \Omega^k(X/\!\!\sim)$
be such that $\pi^* \alpha = 0$.
Then, for any plot $p \colon U \to X$, we have $p^* \pi^* \alpha= 0$.
By the def\/inition of the quotient dif\/feology, this implies that
for any plot $q \colon U \to X/\!\!\sim$ we have $q^* \alpha = 0$.
Hence, $\alpha = 0$, as required.
\end{proof}

\begin{Proposition}[pullbacks from the orbit space are basic]
\label{p:pullbacksarebasic}
Let a Lie group $G$ act on a~mani\-fold~$M$.
Let $\alpha = \pi^*\beta$ for some $\beta \in \Omega^k(M/G)$.
Then $\alpha$ is basic.
\end{Proposition}

\begin{proof}
To show that $\alpha$ is $G$-invariant, note that for every $g \in G$, because $\pi\circ g=\pi$,
\begin{gather*}
g^* \alpha = g^* \pi^* \beta = \pi^* \beta = \alpha.
\end{gather*}

If $\alpha$ is a zero-form (that is, a smooth function)
then $\alpha$ is automatically horizontal and we are done.
Next, we assume that $\alpha$ is a dif\/ferential form of positive degree,
and we show that $\alpha$ is horizontal.
By Lemma~\ref{facts}, if $x\in M$ and $v\in T_x(G\cdot x)$, then there exists
$\xi\in\g$ such that
\begin{gather*}v=\frac{d}{dt}\Big|_{t=0}\exp(t\xi) \cdot x .\end{gather*}
Let $A_x\colon G\to M$ be the map sending $g$ to $g\cdot x$. Then, $v=(A_x)_*(\xi)$.

Thus, to show that $\alpha$ is horizontal, it is enough to show that $A_x^*\alpha=0$ for all $x\in M$. Indeed, it follows from the following commutative diagram that $A_x^*\alpha=A_x^*\pi^*\beta=0$,
\begin{gather*}
\begin{gathered}[b]
\xymatrix{
 G \ar[rr]^{A_x} \ar[d]_{g \mapsto \star} && M \ar[d]^{\pi} \\
 \{ \star \} \ar[rr]_{\star \mapsto [x]} && M/G
}\\[-17pt]
\null
\end{gathered}\tag*{\qed}
\end{gather*}
 \renewcommand{\qed}{}
\end{proof}

\begin{Remark}
\hypref{p:pullbacksarebasic}{Proposition} can also be deduced
from the following lemma:
a dif\/ferential form $\alpha$ on $M$ is basic
if and only if its pullbacks under the maps $G \times M \to M$
given by the projection $(g,x)\mapsto x$
and by the action $(g,x) \mapsto g\cdot x$ coincide.
For a proof of this lemma see, for example,
\cite[Lemma~3.3]{watts-groupoids}.
\end{Remark}

\begin{Proposition}[pullbacks via quotient maps]
\label{p:injectivity}
Let a Lie group $G$ act on a manifold $M$,
and let $\pi \colon M \to M/G$ be the quotient map.
Then the pullback map
\begin{gather*} \pi^* \colon \ \Omega^k(M/G) \to \pi^*\Omega^k(M/G) \end{gather*}
is a diffeomorphism,
where the target space is equipped with the subset diffeology
induced \linebreak from~$\Omega^k(M)$.
\end{Proposition}

\begin{proof}
More generally,
let $X$ be a dif\/feological space, let $\sim$ be an equivalence relation
on $X$, and let $\pi \colon X \to X/\!\!\sim$ be the quotient map.
We will show that the pullback map
\begin{gather*} \pi^* \colon \ \Omega^k(X/\!\!\sim) \to \pi^*\Omega^k(X/\!\!\sim) \end{gather*}
is a dif\/feomorphism,
where the target space is equipped with the subset dif\/feology
induced \linebreak from~$\Omega^k(X)$.

Clearly, $\pi^*$ is surjective to its image.
For the injectivity of~$\pi^*$, see \hypref{l:pullback is injection}{Lemma}.
As noted in Part~\eqref{f:pullback is sm} of \hypref{r:formfacts}{Remark},
$\pi^*$ is smooth.
We wish to show that the inverse map
\begin{gather*} (\pi^*)^{-1} \colon \ \pi^*\Omega^k(X/\!\!\sim) \to \Omega^k(X/\!\!\sim) \end{gather*}
is also smooth.

Fix a plot $p \colon U \to \Omega^k(X)$
with image in $\pi^*\Omega^k(X/\!\!\sim)$.
We would like to show
that $(\pi^*)^{-1} \circ p \colon U \to \Omega^k(X/\!\!\sim)$
is a~plot of $\Omega^k(X/\!\!\sim)$.
By \hypref{d:D on forms}{Def\/inition} of the dif\/feology
on spaces of dif\/ferential forms,
we need to show, given any plot $r \colon W \to X/\!\!\sim$
with $W \subset \RR^n$, that
\begin{gather*} (u,w)\mapsto r^* \big( \big( (\pi^*)^{-1} \circ p \big) (u) \big)\big|_w \end{gather*}
is a map to $\bigwedge^k \RR^n$ that is smooth in $(u,w) \in U \times W$.
Here, $(\pi^*)^{-1}$ is restricted to the image of~$\pi^*$,
on which it is well def\/ined because~$\pi^*$ is injective.

It is enough to show smoothness locally.
For any point $w\in W$ there exist an open neighbourhood $V\subseteq W$ of $w$
and a plot $q \colon V\to X$ such that $r|_V = \pi\circ q$.
For all $v\in V$, we have
\begin{gather*} r^* \big( \big((\pi^*)^{-1}\circ p\big)(u) \big)|_v = q^*(p(u))|_v,\end{gather*}
which is smooth in $(u,v) \in U \times V$
by the def\/inition of the standard functional dif\/feology on~$\Omega^k(X)$.
And so we are done.
\end{proof}

\section{Quotient in stages}\label{sec:q in stages}

In this section we give a technical result that we use in the next section.
All quotients, subsets, and products are assumed to be equipped
with the quotient, subset, and product dif\/feologies.

On the image of any Lie group homomorphism $H \to G$
there exists a unique manifold structure
such that the inclusion map of the image into $G$ is an immersion;
this follows from the second item of Lemma~\ref{facts}.
By \emph{Lie subgroup} of $G$ we refer to such an image.
Thus, Lie subgroups are injectively immersed subgroups
that are not necessarily closed.

Recall that a Lie group $G$ acts \emph{properly} on a manifold $N$
if the map $G\times N\to N\times N$ sending $(g,x)$ to $(x,g\cdot x)$
is proper. The action is said to have \emph{constant orbit-type}
if all stabilisers are conjugate.
If a Lie group acts properly and with a constant orbit type,
then the quotient is a~mani\-fold and the quotient map is a f\/ibre bundle.

\begin{Lemma}[quotient in stages] \label{l:q in stages}
Let a Lie group $G$ act on a manifold $N$.
Let $K$ be a Lie subgroup of~$G$
that is normal in $G$.
Also consider the induced action of $G$ on the quotient~$N/K$.
\begin{enumerate}\itemsep=0pt
\item[$(i)$]
There exists a unique map $ e \colon N/G \to (N/K) / G $
such that the following diagram commutes:
\begin{gather} \label{diagram q in stages}
\begin{split}
   \xymatrix{
 N \ar[rr]^{\pi_K} \ar[d]_{\pi_G} && N/K \ar[d]^{\pi_{G/K}} \\
 N/G \ar[rr]^{e} &&  (N/K)/G
}
\end{split}
\end{gather}

\item[$(ii)$]
The map $e$ is a diffeomorphism.

\item[$(iii)$]
The pullback map
\begin{gather*} \pi_K^* \colon \ \Omega^*(N/K) \to \Omega^*(N) \end{gather*}
restricts to a bijection from $\image \pi_{G/K}^*$ onto $\image \pi_G^*$.

\item[$(iv)$]
Suppose that $K$ acts on $N$ properly and with a constant orbit-type,
so that~$N/K$ is a~mani\-fold and $N \to N/K$ is a fibre bundle.
Then the pullback map $\pi_K^*$ also restricts to a bijection
from $\Omega^*_{\rm basic}(N/K)$ onto $\Omega^*_{\rm basic}(N)$.
Consequently, if one of the inclusions
$\image \pi_{G/K}^* \subset \Omega^*_{\rm basic}(N/K)$
and
$\image \pi_G^* \subset \Omega^*_{\rm basic}(N)$
of Proposition~{\rm \ref{p:pullbacksarebasic}} is an equality,
then so is the other.
\end{enumerate}
\end{Lemma}

\begin{Remark} \label{r:splits}
The $G$-action on $N/K$ factors through an action of $G/K$.
The quotients $(N/K)/G$ and $(N/K)/(G/K)$ coincide,
as they are quotients of $N/K$ by the same equivalence relation.
If $K$ is closed in $G$ (so that $G/K$ is a Lie group)
and $N/K$ is a manifold, then $G$-basic forms coincide with $(G/K)$-basic forms
on $N/K$.
\end{Remark}

\begin{proof}[Proof of Lemma~\ref{l:q in stages}]
Because $\pi_K$ is $G$-equivariant, such a map $e$ exists.
Because $\pi_G$ is onto, such a map $e$ is unique.
Because the preimage under $\pi_K$ of a $G$-orbit in $N/K$
is a single $G$-orbit in $N$, the map $e$ is one-to-one.
Because the maps $\pi_{G/K}$ and $\pi_K$ are onto,
the map $e$ is onto.

Thus, the map $e$ is a bijection. To show that it is a dif\/feomorphism,
it remains to show,
for every parametrisation $p \colon U \to N/G$,
that $p$ is a plot of $N/G$ if and only if $e \circ p$ is a plot of $(N/K)/G$.

Fix a parametrisation,
\begin{gather*} p \colon \ U \to N/G.\end{gather*}

Suppose that $p$ is a plot of $N/G$. Let $u \in U$.
Then there exist an open neighbourhood~$W$ of~$u$ in~$U$
and a plot $q \colon W \to N$ such that $p|_W = \pi_G \circ q$.
The composition $\pi_K \circ q$ is a~plot of~$N/K$, and
\begin{gather*} \pi_{G/K} \circ \pi_K \circ q = e \circ \pi_G \circ q = e \circ p|_W.\end{gather*}
Because $u \in U$ was arbitrary,
this shows that $e \circ p$ is a plot of $(N/K)/G$.

Conversely, suppose that $e \circ p$ is a plot of $(N/K)/G$.
Let $u \in U$. By applying the def\/inition of the quotient dif\/feology
at $\pi_{G/K}$ and then at $\pi_K$,
we obtain an open neighbourhood $W$ of $u$ in $U$
and a plot $r \colon W \to N$ such that
\begin{gather*} e \circ p|_W = \pi_{G/K} \circ \pi_K \circ r .\end{gather*}
By~(\ref{diagram q in stages}) and by the choice of $r$,
\begin{gather*}e \circ \pi_G \circ r = \pi_{G/K} \circ \pi_K \circ r = e \circ p|_W.\end{gather*}

Because $e$ is one-to-one, this implies that $\pi_G \circ r = p|_W$.
Because $u \in U$ was arbitrary, this shows that $p$ is a plot of $N/G$.
This completes the proof that $e$ is a dif\/feomorphism.

Because the diagram~\eqref{diagram q in stages} commutes, $\pi_K^*$
takes $\image \pi_{G/K}^*$ into $\image \pi_G^*$.
Because $e$ is a~dif\/feomorphism, we can consider its inverse.
From the commuting diagram
\begin{gather*} \xymatrix{
 N \ar[rr]^{\pi_K} \ar[d]_{\pi_G} && N/K \ar[d]^{\pi_{G/K}} \\
  N/G &&   (N/K) / G \ar[ll]_{e^{-1}}
} \end{gather*}
we see that $\pi_K^*$ takes $\image \pi_{G/K}^*$ \emph{onto} $\image \pi_G^*$.
Because $\pi_K^*$ is one-to-one (by \hypref{l:pullback is injection}{Lemma}),
we have a bijection
$\pi_K^* \colon \image \pi_{G/K}^* \to \image \pi_G^*$.

Now suppose that $K$ acts on $N$ properly and with a constant orbit-type,
so that $N/K$ is a~mani\-fold and $\pi_K$ is a f\/ibre bundle.
Then we know that $\pi_K^*$ is a bijection from the dif\/ferential forms
on $N/K$ to the $K$-basic dif\/ferential forms on $N$.

Because $\pi_K$ is $G$-equivariant,
$\pi_K^*$ takes $G$-invariant forms on $N/K$
to $G$-invariant forms on $N$
and $G$-horizontal forms on $N/K$ to $G$-horizontal forms on $N$.
So we have an injection
$\pi_K^* \colon \Omega^*_{\rm basic}(N/K) \to \Omega^*_{\rm basic}(N)$.

Let $\alpha$ be a $G$-basic form on $N$. In particular $\alpha$ is $K$-basic,
so there exists a form $\beta$ on $N/K$ such that $\alpha = \pi_K^* \beta$.
Because $\pi_K^*$ is one-to-one and $\alpha$ is $G$-invariant,
$\beta$ is $G$-invariant.
Because $\pi_K$ is $G$-equivariant and $\alpha$ is $G$-horizontal,
$\beta$ is $G$-horizontal.
This completes the proof that the map
$\pi_K^* \colon \Omega^*_{\rm basic}(N/K) \to \Omega^*_{\rm basic}(N)$
is a bijection.
\end{proof}

\section{The pullback surjects onto basics forms} \label{sec:surjects}

The main result of this section is Proposition~\ref{p:geomquot},
in which we give conditions on an action of a Lie group $G$ on a manifold $M$
under which \emph{every} basic form~$\alpha$ is the pullback
of some dif\/feological form on the quotient.
By \hypref{p:pullbackimgchar}{Proposition},
we need to show that $p_1^*\alpha = p_2^*\alpha$ for every two plots
$p_1 \colon U \to M$ and $p_2 \colon U \to M$
such that
for each $u\in U$ there is some $g \in G$ such that $p_2(u) = g \cdot p_1(u)$.
If $g$ can be chosen to be a smooth function of $u$,
then it is easy to conclude that $p_1^*\alpha = p_2^*\alpha$
if $\alpha$ is basic:

\begin{Lemma}\label{l:smooth case}
Let a Lie group $G$ act on a manifold $M$.
Let $p_1, p_2 \colon U \to M$ be plots.
Suppose that $p_2(u) = a(u) \cdot p_1(u)$
for some smooth function $a \colon U \to G$.
Then for every $\alpha \in \Omega^k_{\rm basic}(M)$
we have $p_1^*\alpha = p_2^*\alpha$.
\end{Lemma}

\begin{proof}
Pick a point $u \in U$ and a tangent vector $v \in T_uU$.
Let $\xi_1 = (p_1)_*v $ $\left( \in T_{p_1(u)}M \right)$
and $\xi_2 = (p_2)_*v$ $\left( \in T_{p_2(u)}M \right)$.
Let $g = a(u)$.
The directional derivative $D_v a|_u$ of $a(\cdot)$
in the direction of~$v$ has the form $\eta \cdot g$ $\left( \in T_gG \right)$
for some Lie algebra element~$\eta$ (i.e., it is the right translation
of~$\eta$ by~$g$).
We then have that
\begin{gather*} \xi_2 = g \cdot \xi_1 + {\eta_M}|_{p_2(u)}, \end{gather*}
where $g \cdot \xi_1$ is the image of~$\xi_1$
under the dif\/ferential (push-forward) map
$ g_* \colon T_{p_1(u)}M \to T_{p_2(u)}M $,
and where $\eta_M$ is the vector f\/ield on $M$
that corresponds to~$\eta$;
in particular~$\eta_M$ is everywhere tangent to the $G$ orbits.

Applying this to vectors $v^{(1)}, \ldots, v^{(k)} \in T_uU$,
we get that
\begin{align*}
 (p_2^* \alpha)|_u \big( v^{(1)} , \ldots , v^{(k)} \big)
 &= \alpha|_{ p_2(u) } \big( \xi_2^{(1)} , \ldots ,
 \xi_2^{(k)} \big)
 \quad \text{ where } \xi_2^{(j)} := (p_2)_* v^{(j)} \\
 &= \alpha|_{ p_2(u) }
 \big( g \cdot \xi_1^{(1)} + \eta^{(1)}_M , \ldots ,
 g \cdot \xi_1^{(k)} + \eta^{(k)}_M \big) \\
 & \qquad\qquad\qquad \text{ where } \xi_1^{(j)} := (p_1)_* v^{(j)}
 \text{ and }
 D_{v^{(j)}} a|_u = \eta^{(j)} \cdot g \\
 &= \alpha|_{ p_2(u) } \big( g \cdot \xi_1^{(1)} , \ldots ,
 g \cdot \xi_1^{(k)} \big)
 \quad \text{ because $\alpha$ is horizontal } \\
 &= \alpha|_{ p_1(u) } \big( \xi_1^{(1)} , \ldots ,
 \xi_1^{(k)} \big)
 \quad \text{ because $\alpha$ is invariant } \\
 &= (p_1^* \alpha)|_u \big( v^{(1)} , \ldots , v^{(k)} \big).\tag*{\qed}
\end{align*}
 \renewcommand{\qed}{}
\end{proof}

Unfortunately, in applying \hypref{p:pullbackimgchar}{Proposition},
it might be impossible to choose~$g$
to be a smooth function of~$u$:

\begin{Example}[$\ZZ_2\circlearrowright\RR$]
\label{x:plots differ nonsmoothly}
Let $M=\RR$, let $G = \{ 1, -1 \}$ with $(\pm1) \cdot x = \pm x$,
and let $\pi \colon M \to M/G$ be the quotient map.
Consider the two plots
$p_1 \colon \RR\to M$ and $p_2 \colon \RR\to M$
def\/ined as follows:
\begin{gather*} p_1(t) := \begin{cases} -e^{-1/t^2}  & \text{if} \ \ t<0, \\
 0  & \text{if} \ \ t=0, \\
 e^{-1/t^2}  & \text{if} \ \ t>0,
\end{cases}\end{gather*}
and
\begin{gather*} p_2(t):= \begin{cases} -e^{-1/t^2}  & \text{if} \ \ t\neq 0, \\
 0  & \text{if} \ \ t=0.
\end{cases}\end{gather*}
Then $\pi \circ p_1 = \pi \circ p_2$.
However, for $t<0$ we have $p_1(t)=1\cdot p_2(t)$,
whereas for $t>0$ we have $p_1(t)=-1\cdot p_2(t)$,
and so the two plots
do not dif\/fer by a continuous function to~$G$
on any neighbourhood of $t=0$.
\end{Example}

Our proofs use the following lemma.

\begin{Lemma}\label{l:baire}
Let $U \subseteq \RR^n$ be an open set.
Let $\{ C_i \}$ be a $($finite or$)$ countable collection
of relatively closed subsets of $U$ whose union is $U$.
Then the union of the interiors, $\bigcup_i \inter(C_i)$,
is open and dense in~$U$.
\end{Lemma}

\begin{proof}
This is a consequence of the Baire category theorem.
\end{proof}

%%% Countable group action case
%%% and reduction of compact Lie group to identity component

We now prove Proposition~\ref{p:geomquot} in the special case of a f\/inite group action. In this case, basic dif\/ferential forms are simply invariant dif\/ferential forms, as the tangent space to an orbit at any point is trivial.

\begin{Proposition}[case of a f\/inite group] \label{p:finite}
Let $G$ be a finite group, acting on a manifold~$M$.
Then every basic form on~$M$ is the pullback of a~$($diffeological$)$ differential form on~$M/G$.
\end{Proposition}

\begin{proof}
Fix a basic dif\/ferential $k$-form $\alpha$ on~$M$.
By \hypref{p:pullbackimgchar}{Proposition},
it is enough to show the following:
if $p_1 \colon U \to M$ and $p_2 \colon U \to M$ are plots
such that $\pi \circ p_1 = \pi \circ p_2$,
then $p_1^* \alpha = p_2^*\alpha$ on~$U$.
Fix two such plots $p_1 \colon U \to M$ and $p_2 \colon U \to M$.
For each $g \in G$ let
\begin{gather*} C_g := \{ u \in U\,|\, g \cdot p_1(u) = p_2(u) \}.\end{gather*}
By continuity, $C_g$ is closed for each~$g$.
By our assumption on~$p_1$ and~$p_2$,
\begin{gather*} U = \bigcup_{g\in G} C_g. \end{gather*}
By \hypref{l:baire}{Lemma},
the set $\bigcup_{g \in G} \inter(C_g)$
is open and dense in $U$. Thus, by continuity, it is enough to show
that $p_1^* \alpha = p_2^*\alpha$ on $\inter(C_g)$ for each $g \in G$. This, in turn, follows from the facts that, for each $g\in G$, we have $g\circ p_1=p_2$ on $\inter(C_g)$ and $g^*\alpha=\alpha$.
\end{proof}

Our next result, contained in \hypref{l:conntodisc}{Lemma}
and preceded by \hypref{l:C closed}{Lemma},
is a generalisation of the case of a f\/inite group action:
it shows that the property that interests us
holds for a Lie group action
if it holds for the action of the identity component of that
Lie group, assuming that the action of the identity component is proper.

\begin{Lemma}\label{l:C closed}
Let $G$ be a Lie group, and let $G_0$ be its identity component.
Assume that $G$ acts on a manifold $M$ such that the restricted action
of $G_0$ on $M$ is proper. Then, for any $\gamma\in G/G_0$,
and for any two plots $p_1 \colon U \to M$ and $p_2 \colon U \to M$, the set
\begin{gather*} C_\gamma := \{ u \in U\,|\, \exists\, g \in\gamma\text{ such that }
 g\cdot p_1(u)=p_2(u)\}\end{gather*}
is $($relatively$)$ closed in $U$.
\end{Lemma}

\begin{proof}
Because the $G_0$ action on $M$ is proper, the set
\begin{gather*}
\Delta := \{ (m,m') \in M \times M \, | \,
\exists\, g_0 \in G_0 \text{ such that } g_0 \cdot m = m' \},
\end{gather*}
being the image of the proper map $(g_0,m) \mapsto (m,g_0 \cdot m)$,
is closed in $M \times M$.

Fix $g' \in \gamma$. Because $G_0$ is normal in $G$,
we can express $\gamma$ as the left coset $G_0 g'$, and we have
\begin{gather*} C_\gamma = \{ u \in U \, | \, \exists\, g_0 \in G_0 \text{ such that }
g_0 g' \cdot p_1(u) = p_2(u) \} .\end{gather*}
We conclude by noting that $C_\gamma$ is the preimage of the closed set $\Delta$
under the continuous map $U \to M \times M$
given by $u \mapsto ( g' \cdot p_1(u) , p_2(u) )$.
\end{proof}

\begin{Lemma} \label{l:conntodisc}
Let $G$ be a Lie group. Let $G_0$ be the identity component of~$G$.
Fix an action of~$G$ on a~mani\-fold~$M$.
Suppose that the restricted $G_0$-action is proper,
and suppose that every $G_0$-basic differential form on~$M$
is the pullback of a diffeological form on $M/G_0$.
Then every $G$-basic differential form on~$M$
is the pullback of a diffeological form on~$M/G$.
\end{Lemma}

\begin{proof}
Fix a $G$-basic form $\alpha$ on $M$.
Let $\pi \colon M\to M/G$ be the quotient map,
and let $p_1 \colon U\to M$ and $p_2 \colon U\to M$ be plots
such that $\pi\circ p_1=\pi\circ p_2$.
Fix $\gamma\in G/G_0$ and
$g' \in \gamma$, and def\/ine $C_\gamma$
as in \hypref{l:C closed}{Lemma}. Def\/ine $\tilde{p}_1 \colon U \to M$
as the composition
$g' \circ p_1$. This is a plot of $M$,
and for any $u \in \inter(C_\gamma)$ we have
$p_2(u) = g_0 \cdot\tilde{p}_1(u)$ for some $g_0 \in G_0$.
Consider the restricted action of $G_0$ on $M$.
Let $\pi_0 \colon M \to M/G_0$ be the corresponding quotient map.
Then the restrictions $\tilde{p}_1|_{\inter(C_\gamma)}$
and $p_2|_{\inter(C_\gamma)}$ are plots of $M$,
and they satisfy
$\pi_0 \circ \tilde{p}_1|_{\inter(C_\gamma)}
 = \pi_0 \circ p_2|_{\inter(C_\gamma)}$.
By hypothesis, and because $\alpha$ is $G_0$-basic (as it is $G$-basic),
$\alpha$ is a pullback of a dif\/feological form on $M/G_0$.
By \hypref{p:pullbackimgchar}{Proposition},
$\tilde{p}_1^* \alpha = p_2^*\alpha$ on $\inter(C_\gamma)$.

But on $\inter(C_\gamma)$ we have
$\tilde{p}_1^* \alpha = p_1^* {g'}^* \alpha = p_1^* \alpha$
(since $\alpha$ is $G$-invariant),
and so $p_1^* \alpha = p_2^* \alpha$ on $\inter(C_\gamma)$.
Since $\gamma\in G/G_0$ is arbitrary,
and $\bigcup_{\gamma\in G/G_0}\inter(C_\gamma)$ is open and dense in $U$
by Lemmas~\ref{l:baire} and \ref{l:C closed},
from continuity we have that $p_1^*\alpha=p_2^*\alpha$ on all of $U$.
Finally, by \hypref{p:pullbackimgchar}{Proposition},
$\alpha$ is the pullback of a form on $M/G$.
\end{proof}

%%% Pointwise result at origin for connected compact linear orthogonal actions.

We proceed with two technical lemmas that we will use
to handle non-trivial compact connected stabilisers.

\begin{Lemma}\label{l:techn1}
Let $G$ be a compact connected Lie group acting orthogonally
on some Euclidean space $V = \RR^N$.
Let $g \in G$ and $\eta \in \g$ be such that $\exp(\eta) = g$.
Let $v \in V$. Then there exists $v' \in V$
such that $|v'| \leq |v|$ and $g \cdot v - v = \eta \cdot v'$.
\end{Lemma}

\begin{proof}
Since $V$ is a vector space, we identify tangent spaces at points of $V$
with $V$ itself.
We also identify elements of the group $G$ and of the Lie algebra $\g$
with the matrices by which these elements act on $V=\RR^N$
\begin{gather*}
 g \cdot v - v = \exp(t\eta) \cdot v \big|_0^1\\
\hphantom{g \cdot v - v}{}
 = \int_0^1 \left( \frac{d}{dt} \exp(t\eta) \cdot v\right) dt\\
\hphantom{g \cdot v - v}{}
 = \int_0^1 \left( \eta \cdot \exp(t\eta) \cdot v\right) dt \\
 \hphantom{g \cdot v - v}{}
 = \eta \cdot \int_0^1 \left( \exp(t\eta) \cdot v \right) dt.
\end{gather*}
So def\/ine $v':=\int_0^1\left(\exp(t\eta)\cdot v\right)dt$. Finally,
\begin{gather*}
 |v'| = \left| \int_0^1 \left( \exp(t\eta) \cdot v \right) dt \right|\\
\hphantom{|v'|}{}
 \leq \int_0^1 \left| \exp(t\eta) \cdot v\right| dt\\
 \hphantom{|v'|}{}
 = \int_0^1 |v| dt \quad \text{because the action is orthogonal} \\
\hphantom{|v'|}{} = |v|.
\end{gather*}
This completes the proof.
\end{proof}

\begin{Lemma}\label{l:techn2}
Let $G$ be a compact connected Lie group acting orthogonally
on some Euclidean space $V = \RR^N$.
Let $\gamma_1$ and $\gamma_2$ be smooth curves from $\RR$ into $V$
such that $ \gamma_1(0) = \gamma_2(0) = 0 $
and such that for every $t \in \RR$ there exists $g_t \in G$
satisfying $\gamma_2(t) = g_t \cdot \gamma_1(t)$.
Let $\xi_1 = \dot{\gamma}_1(0)$ and $\xi_2 = \dot{\gamma}_2(0)$.
Then, for every horizontal form $\alpha$ on $V$,
we have $(\xi_2-\xi_1)\hook\alpha|_0=0$.
\end{Lemma}

Note that, in the assumptions of this lemma,
$t \mapsto g_t$ is not necessarily continuous.

\begin{proof}
We claim that there exists a sequence of vectors $v'_n$
converging to $0$ in $V$, and a sequence~$\mu_n$ in~$\g$, such that
\begin{gather*} \xi_2 - \xi_1 = \underset{n\to\infty}{\lim} \mu_n\cdot v'_n. \end{gather*}

Indeed,
choose any sequence of non-zero real numbers $t_n$ converging to $0$.
For
each $n$ choose $\eta_n \in \g$ such that $\exp(\eta_n)=g_{t_n}$.
Since we are working on a vector space, we can subtract the curves
and consider $\gamma_2(t) - \gamma_1(t)$. We have
\begin{gather*}
 \xi_2 - \xi_1 = \frac{d}{dt} \Big|_{t=0} (\gamma_2(t)-\gamma_1(t))\\
  \hphantom{\xi_2 - \xi_1}{}
 = \underset{t \to 0} {\lim}
 \left( \frac{\gamma_2(t)-\gamma_1(t)}{t} \right)\\
 \hphantom{\xi_2 - \xi_1}{}
 = \underset{n\to\infty} {\lim}
 \left( \frac{\gamma_2(t_n)-\gamma_1(t_n)}{t_n} \right) \\
 \hphantom{\xi_2 - \xi_1}{}
 = \underset{n\to\infty}{\lim}
 \left( \frac{g_{t_n} \cdot \gamma_1(t_n) - \gamma_1(t_n)}{t_n} \right)\\
 \hphantom{\xi_2 - \xi_1}{}
 = \underset{n\to\infty}{\lim}
 \left( \frac{ \eta_n \cdot v'_n }{t_n} \right)
\end{gather*}
for some $\eta_n \in \g$ and
for some $v'_n \in V$ that satisfy $ |v'_n| \leq |\gamma_1(t_n)| $;
the last equality is a result of \hypref{l:techn1}{Lemma}.
Because $|v_n'| \leq |\gamma_1(t_n)| \underset{n\to\infty}{\to} |\gamma_1(0)|
 = 0$,
the claim holds with $\mu_n := \eta_n/t_n$.

We now have
\begin{gather*}
 (\xi_2-\xi_1) \hook \alpha|_0 = \underset{n\to\infty}{\lim}
 \left( (\mu_n\cdot v'_n) \hook(\alpha|_{v'_n}) \right)\\
 \hphantom{(\xi_2-\xi_1) \hook \alpha|_0}{}
 = \underset{n\to\infty}{\lim}
 \left( (\mu_n)_V\hook\alpha|_{v'_n} \right),
\end{gather*}
where $(\mu_n)_V$ is the vector f\/ield on $V$ induced by $\mu_n\in\g$.
Because $\alpha$ is horizontal, the last term above vanishes.
\end{proof}

The f\/inal ingredient that we need for Proposition~\ref{p:geomquot} is
the following property of the model that appears in the slice theorem
(Theorem~\ref{t:slice theorem}).
Let $G$ be a Lie group, $H$ a closed subgroup, and $V$ a vector space
with a linear $H$-action. Recall that the equivariant vector bundle
$ G \times_H V $
over $G/H$ is obtained as the quotient of $G \times V$
by the anti-diagonal $H$-action $h \cdot (g,v) = (g h^{-1} , h \cdot v)$
and that the $G$-action on $G\times_H V$ is $g\cdot[g',v]=[gg',v]$.

\begin{Lemma} \label{l:reduce to slice}
Suppose that every $H$-basic form on $V$ is the pullback
of a diffeological form on~$V/H$.
Then every $G$-basic form on $G \times_H V$ is the pullback
of a diffeological form on \mbox{$(G \times_H V)/G$}.
\end{Lemma}

\begin{proof}
Let $G \times H$ act on $G \times V$
where $G$ acts by left multiplication on the f\/irst factor
and where $H$ acts by the anti-diagonal action
$ h \colon (g,v) \mapsto (gh^{-1} , h \cdot v)$.
We have two maps:
\begin{gather*} %\label{d:triangle}
\xymatrix{
 & G \times V \ar[dl]_{\pr_2} \ar[dr]^{\pi_H} & \\
 V & &  \quad G \times_H V
}
\end{gather*}
Here, the map $\pi_H \colon G \times V \to G \times_H V$ is the quotient
by the $H$-action, and the projection to the second factor
$\pr_2 \colon G \times V \to V$
can be identif\/ied with the quotient by the $G$-action.

We apply quotient in stages (\hypref{l:q in stages}{Lemma})
in two ways:
taking the quotient
by~$G$ and then by~$H$, and taking the quotient by~$H$ and then by~$G$.
This gives the following commuting diagram:
\begin{gather} \label{two quotients}
\begin{split}
&
\xymatrix{
 & G \times V \ar[dl]_{\pr_2} \ar[dd]^{\pi_{G \times H}} \ar[dr]^{\pi_H} & \\
 V \ar[d]_{\pi_V} & & G \times_H V \ar[d]^{\pi} \\
   V/H \ar[r]^(0.3){e} & (G \times V) / (G \times H)
 &  (G \times_H V)/G \ar[l]_(0.42){e'}
}
\end{split}
\end{gather}
where $e$ and $e'$ are dif\/feomorphisms, and $\pi_V$, $\pi_{G\times H}$, and $\pi$ are the quotient maps.

By hypothesis, the inclusion
$\image \pi_V^* \subset \Omega^*_{\text{basic}}(V)$ is an equality.
Applying Part (iv) of \hypref{l:q in stages}{Lemma} to the left hand side
of the diagram~\eqref{two quotients},
we conclude that the inclusion
$\image \pi_{G \times H}^* \subset
 \Omega^*_{\text{basic}}(G \times V)$ is an equality.
Applying Part (iv) of \hypref{l:q in stages}{Lemma} to the right hand side
of the diagram~\eqref{two quotients},
we further conclude that the inclusion
$\image \pi^* \subset \Omega^*_{\text{basic}} (G \times_H V)$
is an equality.
\end{proof}

We are now ready to prove the main result of this section.

\begin{Proposition}[pullback surjects to basic forms]
\label{p:geomquot}
Let a Lie group $G$ act on a~mani\-fold~$M$.
Assume that the identity component of~$G$ acts properly.
Let $\pi \colon M \to M/G$ be the quotient map.
Then every basic form on $M$ is the pullback of a $($diffeological$)$ differential
form on $M/G$ via $\pi^*$.
\end{Proposition}

\begin{proof}
By \hypref{l:conntodisc}{Lemma}, to prove this result
for an arbitrary Lie group, it is enough to prove it
for the action of the identity component of the group.

For every non-negative integer $d$, consider the following two statements.
\begin{enumerate}\itemsep=0pt
\item[A(d):]
For every connected Lie group $K$ with $\dim K = d$,
and for every $K$-manifold $N$ on which the $K$-action is proper,
every $K$-basic form on~$N$ is the pullback of a dif\/ferential form on~$N/K$.

\item[B(d):]
For every Lie group $K$ with $\dim K = d$,
and for every $K$-manifold $N$ on which the identity component of $K$
acts properly, every $K$-basic form on~$N$ is the pullback
of a dif\/ferential form on $N/K$.
\end{enumerate}

Since the result holds for the trivial group, Statement A(0) is true.
By \hypref{l:conntodisc}{Lemma}, it follows that Statement B(0) is true.
Proceeding by induction, we f\/ix a positive integer $d$,
we assume that Statement~B(d$'$) is true for all $d' < d$,
and we would like to prove that Statement~B(d) is true.
By \hypref{l:conntodisc}{Lemma}, Statement~A(d) implies Statement~B(d);
thus, it is enough to prove that Statement~A(d) is true.
That is, we may now restrict to the special case of the proposition
in which the group is connected and the action is proper,
while assuming that the general case of the proposition is true
for all Lie groups of smaller dimension.

Now, let $G$ be a connected Lie group,
and let $M$ be a $G$-manifold on which the $G$-action is proper.
Fix a $G$-basic form~$\alpha$ on~$M$.
We would like to show that~$\alpha$ is the pullback of a~dif\/ferential form
on~$M/G$.

By \hypref{p:pullbackimgchar}{Proposition} we need to show,
for any two plots $p_1 \colon W \to M$
and $p_2 \colon W\to M$ for which $\pi \circ p_1 = \pi \circ p_2$,
that $p_1^* \alpha = p_2^* \alpha$.
Let $p_1$ and $p_2$ be two such plots.
Fix $u\in W$. We would like to show that $p_1^*\alpha|_u = p_2^*\alpha|_u$.

Let $x = p_2(u)$. Let $H$ be the stabiliser of $x$.
By Theorem~\ref{t:slice theorem} %the \hyperref[t:slice theorem]{slice theorem}
there exists
a $G$-invariant open neighbourhood $U$ of $x$ and an equivariant
dif\/feomorphism $F \colon U \to G \times_H V$
where $V = T_xM / T_x(G \cdot x)$.
Because $F$ is an equivariant dif\/feomorphism and $\alpha$ is $G$-basic,
$(F^{-1})^* \alpha$ is $G$-basic on $G \times_H V$.

Either $x$ is a f\/ixed point, or $x$ is not a f\/ixed point.

Suppose that $x$ is a f\/ixed point. Then $H=G$.
So $p_1(u) = p_2(u) = x$, and $F$ identif\/ies $U$ with $V = T_xM$,
sending $x$ to $0 \in V$.
Fixing a $G$-invariant Riemannian metric, we have
that $G$ acts linearly and orthogonally on $V$. Let $v \in T_uW$. Applying \hypref{l:techn2}{Lemma} to the curves
$\gamma_1(t) := F(p_1(u+tv))$ and $\gamma_2(t) := F(p_2(u+tv))$ in $V$
and to the basic form $(F^{-1})^*\alpha$ on $V$,
we obtain that $\dot{\gamma}_1(0) \hook (F^{-1})^*\alpha|_0
 = \dot{\gamma}_2(0) \hook (F^{-1})^*\alpha|_0$.
This, in turn, implies that $v \hook p_1^*\alpha|_u = v \hook p_2^* \alpha|_u$.
Because $v \in T_uW$ is arbitrary, we conclude
that $p_1^*\alpha|_u = p_2^*\alpha|_u$, as required.

Suppose that $x$ is not a f\/ixed point.
Then the stabiliser $H$ of $x$ is a proper subgroup of $G$.
Since $G$ is connected, $\dim H < \dim G$.
By the induction hypothesis, every $H$-basic form on $V$
is the pullback of a dif\/ferential form on $V/H$.
By \hypref{l:reduce to slice}{Lemma},
every $G$-basic form on $G \times_H V$
is the pullback of a dif\/ferential form on $(G \times_H V)/G$.
Because $F$ is an equivariant dif\/feomorphism,
every $G$-basic form on $U$
is the pullback of a dif\/ferential form on $U/G$.
So $\alpha|_U$ is the pullback of a dif\/ferential form on $U/G$.
This implies that $p_1^* \alpha|_u = p_2^* \alpha|_u$, as required.
\end{proof}

\begin{proof}[Proof of Theorem \ref{t:main}]
By \hypref{l:pullback is injection}{Lemma}
and \hypref{p:pullbacksarebasic}{Proposition},
the pullback is an injection into the space of basic forms.
By \hypref{alg iso}{Remark} and \hypref{p:injectivity}{Proposition},
as a map to its image,
the pullback is an isomorphism of dif\/ferential graded algebras
and a dif\/feological dif\/feomorphism.
By \hypref{p:geomquot}{Proposition}, if the identity component of $G$
acts properly, the image is the space of basic forms.
\end{proof}

\begin{Example}[irrational torus, f\/irst construction]
\label{x:irrational torus}
Fix an irrational number $\alpha\in\RR{\smallsetminus}\QQ$.
The corresponding \emph{irrational torus} is
\begin{gather*} T_\alpha := \RR / (\ZZ + \alpha \ZZ). \end{gather*}
It is obtained as the quotient of $\RR$
by the $\ZZ^2$-action $(m,n) \cdot x = x + m + n \alpha$;
note that it is not Hausdorf\/f.
The basic dif\/ferential forms on $\RR$ with respect to this action
are the constant functions and the constant coef\/f\/icient one-forms~$cdx$.
By \hypref{p:geomquot}{Proposition}, each of these
is the pullback of a dif\/ferential form on~$T_\alpha$.
\end{Example}

\begin{Remark} \label{rk:irrational torus}
In Example~\ref{x:irrational torus},
although the topology of $T_\alpha$ is trivial,
its de Rham cohomology is isomorphic to that of a circle.
We note, though, that dif\/ferential forms still do not capture
the richness of the dif\/feology on $T_\alpha$:
by Donato and Iglesias \cite{DoIg},
$T_\alpha$ and $T_\beta$ are dif\/feomorphic
if and only if there exist integers $a$, $b$, $c$, $d$
such that $ad-bc = \pm 1$ and $\alpha = \frac{a + \beta b}{c + \beta d}$.
See Exercise 4 and Exercise 105 of \cite{iglesias}
with solutions at the end of the book.
\end{Remark}

\begin{Example}[irrational torus, second construction]
\label{x:irrational solenoid}
Fix an irrational number $\alpha\in\RR{\smallsetminus}\QQ$.
Consider the quotient
\begin{gather*} \TT^2 / S_\alpha \end{gather*}
of $\TT^2 := \RR^2 / \ZZ^2$ by the \emph{irrational solenoid}
$S_\alpha := \{ [t,\alpha t] \, | \, t \in \RR \} \subset \TT^2$.
It is obtained as the quotient of $\TT^2$ by the $\RR$-action
$t \cdot [x,y] = [x + t , y + \alpha t]$.
The basic forms on $\TT^2$ with respect to this action
are the constant functions and the constant multiples
of the one-form $\alpha dx - dy$.
The quotient $\TT^2/S_\alpha$ is dif\/feomorphic
to the irrational torus $T_\alpha$ of \hypref{x:irrational torus}{Example};
see Exercise~31 of~\cite{iglesias} (with solution at the end
of the book).

In fact, consider the action of $\RR \times \ZZ^2$ on $\RR^2$
that is given by
$(t,m,n) \cdot (x,y) = (x + m + t , y + n + \alpha t)$.
Taking the quotient f\/irst by $\RR$ and then by $\ZZ^2$
(and identifying the f\/irst of these quotients with $\RR$
through the map $(x,y) \mapsto y - \alpha x$)
yields~$T_\alpha$.
Taking the quotient f\/irst by~$\ZZ^2$ and then by~$\RR$ yields $\TT^2/S_\alpha$.
Applying \hypref{l:q in stages}{Lemma} twice,
we get the following commuting diagram:
\begin{gather*} \xymatrix{
 & \RR^2 \ar[dl] \ar[dr] \ar[dd] & \\
 \RR \ar[d] & & \TT^2 \ar[d] \\
T_\alpha \ar[r]^(.3){e} &  \RR^2 / (\RR \times \ZZ^2)
 &   \TT^2/S_\alpha \ar[l]_(.35){e'}
} \end{gather*}
where $e$ and $e'$ are dif\/feomorphisms.
As noted in \hypref{x:irrational torus}{Example},
every basic form on $\RR$ is the pullback of a dif\/feological
dif\/ferential form on $T_\alpha$. By \hypref{l:q in stages}{Lemma},
this implies that every basic form on $\RR^2$
is the pullback of a dif\/feological dif\/ferential form
on $\RR^2/(\RR \times \ZZ^2)$.
Again by \hypref{l:q in stages}{Lemma},
we conclude that every basic form on $\TT^2$
is the pullback of a dif\/feological form on $\TT^2/S_\alpha$.
Thus, the $\RR$-action on $\TT^2$ through $S_\alpha$
satisf\/ies the conclusion of \hypref{p:geomquot}{Proposition},
although it does not satisfy the assumption
of \hypref{p:geomquot}{Proposition}: this $\RR$-action is not proper.
\end{Example}

\appendix
%%%%%%%%%%%%%%%%%%%%%%%%%%%%%%%%%%%%%%%%%%%%%%%%%%%%%%%%%%%%%%%%%%%%%%%%%%%%%%%
\section{Orbifolds}\label{sec:orbifolds}
%%%%%%%%%%%%%%%%%%%%%%%%%%%%%%%%%%%%%%%%%%%%%%%%%%%%%%%%%%%%%%%%%%%%%%%%%%%%%%%

Let $X$ be a Hausdorf\/f, second countable topological space.
Fix a positive integer $n$.

The following def\/inition is based on Haef\/liger, \cite[Section~4]{haefliger}.

\begin{enumerate}\itemsep=0pt
\item
An $n$ dimensional \emph{orbifold chart} on $X$ is a triple
$(\tilde{U},\Gamma,\phi)$ where $\tilde{U}\subseteq\RR^n$ is an open ball,
$\Gamma$ is a f\/inite group of dif\/feomorphisms of $\tilde{U}$,
and $\phi \colon \tilde{U}\to X$ is a $\Gamma$-invariant map
onto an open subset $U$ of $X$
that induces a homeomorphism $\tilde{U}/\Gamma \to U$.

\item
Two orbifold charts on $X$,
$(\tilde{U},\Gamma,\phi)$ and $(\tilde{V},\Gamma',\psi)$,
are \emph{compatible}
if for every two points $u \in \tilde{U}$ and $v \in \tilde{V}$
such that $\phi(u) = \psi(v)$ there exist
a neighbourhood $O_u$ of $u$ in $\tilde{U}$
and a neighbourhood $O_v$ of $v$ in $\tilde{V}$
and a dif\/feomorphism $g \colon O_u \to O_v$
that takes $u$ to $v$ and such that $\psi \circ g = \phi$.

\item \looseness=-1
An \emph{orbifold atlas} on $X$ is a set of orbifold charts on $X$
that are pairwise compatible and whose images cover $X$.
Two orbifold atlases are \emph{equivalent} if their union is an orbifold atlas.
\end{enumerate}

The following def\/inition was introduced in \cite{IZKZ}:
A \emph{diffeological orbifold} is a dif\/feological space
that is locally dif\/feomorphic to f\/inite linear quotients of $\RR^n$.

These two def\/initions are equivalent in the following sense.
Given an orbifold atlas on~$X$, there exists a unique dif\/feology on $X$
such that all the homeomorphisms $\tilde{U}/\Gamma \to U$
are dif\/feomor\-phisms. With this dif\/feology, $X$ becomes a dif\/feological
orbifold. Two orbifold atlases are equivalent if and only if
the corresponding dif\/feologies are the same.
Finally, every dif\/feological orbifold structure on $X$
can be obtained in this way.
For details, see \cite[Section~8]{IZKZ}.

Orbifolds were initially introduced by Ichiro Satake \cite{satake56,satake57}
under the name ``V-manifolds''.
Satake's approach is equivalent to Haef\/liger's; see~\cite{IZKZ}.
Satake \cite{satake57} also introduced tensors,
and in particular dif\/ferential forms, on V-manifolds.
Haef\/liger's approach yields the following def\/inition.

Let $\{ (\tU, \Gamma, \psi) \} $ be an orbifold atlas on $X$.
An \emph{orbifold differential form} on $X$
is given by, for each chart $(\tU, \Gamma, \psi)$ in the atlas,
a $\Gamma$-invariant dif\/ferential form $\alpha_{\tU}$ on the domain $\tU$
of the chart. We require the following \emph{compatibility condition}.
For every two charts $(\tU,\Gamma,\phi)$, $(\tilde{V},\Gamma',\psi)$,
and every two points $u \in \tU$ and $v \in \tilde{V}$
with $\phi(u) = \psi(v)$,
there exist a dif\/feomorphism $g \colon O_u \to O_v$
from a neighbourhood of $u$ to a neighbourhood of $v$
that takes $u$ to $v$, such that $\psi \circ g = \phi$,
and such that $g^* (\alpha_{\tilde{V}}|_{O_v}) = \alpha_{\tU}|_{O_u}$.
Two such collections $\{ \alpha_{\tU} \}$ of dif\/ferential forms,
def\/ined on the domains of the charts in two equivalent orbifold atlases,
represent the same orbifold dif\/ferential form
if their union still satisf\/ies the compatibility condition.

Every dif\/feological dif\/ferential form $\alpha$ on $X$
determines an orbifold dif\/ferential form
by associating to every chart $(\tU, \Gamma, \psi)$
the pullback $\psi^* \alpha$.
\hypref{p:finite}{Proposition} implies that this gives a~bijection
between dif\/feological dif\/ferential forms and orbifold dif\/ferential forms.

\section[Sjamaar dif\/ferential forms; case of regular symplectic quotients]{Sjamaar dif\/ferential forms;\\ case of regular symplectic quotients}
\label{sec:sjamaar}

Let a Lie group $G$ act properly on a symplectic manifold $(M,\omega)$
with an (equivariant) momentum map $ \Phi \colon M \to \g^*$.
Let $Z = \Phi^{-1}(0)$ be the zero level set
and $i \colon Z \to M$ its inclusion map.
Let
\begin{gather*}
 Z_\reg = \{ z \in Z \, | \, \text{$\exists$ neighbourhood $U$
of $z$ in $Z$ such that, for all $z' \in U$, } \\
\hphantom{Z_\reg = \{ z \in Z \, |\, \text{$\exists$ neighbourhood $U$ }}{}
 \text{the stabilisers of $z'$ and of $z$ are conjugate in $G$} \}.
\end{gather*}
The set $Z_\reg$, (with the subset dif\/feology induced from~$M$
or, equivalently, from~$Z$) is a manifold,
and it is open and dense in~$Z$ (see \cite{SL}).
The quotient $Z_\reg/G$, (with the quotient dif\/feology induced from $Z_\reg$,
or, equivalently, the subset dif\/feology induced from~$M/G$,)
is also a~manifold.

Above, the connected components of $Z_\reg$ and $Z_\reg/G$
may have dif\/ferent dimensions.
If $M$ is connected and $\Phi$ is proper, then $Z_\reg$ and $Z_\reg/G$
are connected. See, for example, \cite{LMTW}~and~\cite{SL}.

Denote by $i_\reg \colon Z_\reg \to M$ the inclusion map
and by $\pi_\reg \colon Z_\reg \to Z_\reg/G$ the quotient map.

The following def\/inition was introduced (but not yet named)
by Reyer Sjamaar in \cite{sjamaar}:

\begin{Definition}
A \emph{Sjamaar differential $l$-form} $\sigma$ on $Z/G$
is a dif\/ferential $l$-form on $Z_\reg/G$ (in the ordinary sense)
such that there exists $\tilde{\sigma}\in\Omega^l(M)$
satisfying $i_{\reg}^*\tilde{\sigma}=\pi_{\reg}^*\sigma$.
\end{Definition}

\looseness=-1
A special case of a Sjamaar form is the \emph{reduced symplectic form},
$\omega_\red$,
which satisf\/ies $\pi_\reg^* \omega_\red = i_\reg^* \omega$.
The orbit type stratif\/ication on $M$
induces a stratif\/ication of the reduced space $Z/G$,
and the Sjamaar dif\/ferential forms naturally extend to the strata of $Z/G$.
The extensions of $\omega_\red$ to these strata exhibit $Z/G$
as a \emph{stratified symplectic space}
in the sense of Sjamaar and Lerman~\cite{SL}.

The space of Sjamaar forms is closed under wedge products
and forms a subcomplex
of the de Rham complex $(\Omega^*(Z_\reg/G),d)$.
Sjamaar forms satisfy a Poincar\'e lemma, Stokes' theorem,
and a de Rham theorem.

For details, see Sjamaar's paper \cite{sjamaar}.

The reduced space $Z/G$ comes equipped with the quotient dif\/feology inherited from~$Z$, which equals the subset dif\/feology inherited from~$M/G$. We call this
the \emph{subquotient} dif\/feology.

It is now natural to ask how Sjamaar forms on a symplectic quotient~$Z/G$,
which \emph{a-priori} depend on the ambient symplectic manifold~$M$,
relate to the dif\/feological forms on $Z/G$, whose def\/inition is intrinsic.
More precisely, consider the inclusion map
$J \colon Z_\reg/G \to Z/G$. Then we have the pullback map
on dif\/feological forms
\begin{gather*} J^* \colon \ \Omega^l(Z/G) \to \Omega^l(Z_\reg/G) ,\end{gather*}
and we identify the target space with the ordinary dif\/ferential forms
on $Z_\reg/G$. We ask:
\begin{itemize}\itemsep=0pt
\item
Is the space of Sjamaar forms contained in the image of $J^*$?
\item
Is the image of $J^*$ contained in the space of Sjamaar forms?
\item
Is $J^*$ one-to-one?
\end{itemize}

If $0$ is a regular value of the momentum map~$\Phi$,
then it follows from \hypref{p:finite}{Proposition}
that the answers to each of these questions is ``yes''.
If $0$ is a critical value, then the answer to the f\/irst question is ``yes'',
and we do not know the answers to the other two questions.
We refer the reader to Section~3.4 of the second author's thesis~\cite{watts}
for details.

\subsection*{Acknowledgements}
This work is partially supported by the Natural Sciences
and Engineering Council of Canada.
We are grateful to Patrick Iglesias-Zemmour for instructing us on dif\/feology
and to Reyer Sjamaar for his inspiration, as well as to the anonymous referees for excellent suggestions that lead to a~better organisation of the paper.

\pdfbookmark[1]{References}{ref}
\LastPageEnding


\begin{thebibliography}{99}
\footnotesize\itemsep=0pt

\bibitem{BH}
Baez J.C., Hof\/fnung A.E., Convenient categories of smooth spaces,
 \href{http://dx.doi.org/10.1090/S0002-9947-2011-05107-X}{\textit{Trans. Amer. Math. Soc.}} \textbf{363} (2011), 5789--5825,
 \href{http://arxiv.org/abs/0807.1704}{arXiv:0807.1704}.

\bibitem{BFW}
Blohmann C., Fernandes M.C.B., Weinstein A., Groupoid symmetry and constraints
 in general relativity, \href{http://dx.doi.org/10.1142/S0219199712500617}{\textit{Commun. Contemp. Math.}} \textbf{15} (2013),
 1250061, 25~pages, \href{http://arxiv.org/abs/1003.2857}{arXiv:1003.2857}.

\bibitem{bredon}
Bredon G.E., Introduction to compact transformation groups, \textit{Pure and Applied Mathematics}, Vol.~46, Academic
 Press, New York~-- London, 1972.

\bibitem{chen1}
Chen K.-T., Iterated integrals of dif\/ferential forms and loop space homology,
 \href{http://dx.doi.org/10.2307/1970846}{\textit{Ann. of Math.}} \textbf{97} (1973), 217--246.

\bibitem{chen4}
Chen K.-T., On dif\/ferentiable spaces, in Categories in Continuum Physics
 ({B}uf\/falo, {N}.{Y}., 1982), \href{http://dx.doi.org/10.1007/BFb0076932}{\textit{Lecture Notes in Math.}}, Vol.~1174,
 Springer, Berlin, 1986, 38--42.

\bibitem{CrSt}
Crainic M., Struchiner I., On the linearization theorem for proper {L}ie
 groupoids, \textit{Ann. Sci. \'Ec. Norm. Sup\'er.~(4)} \textbf{46} (2013),
 723--746, \href{http://arxiv.org/abs/1103.5245}{arXiv:1103.5245}.

\bibitem{CuSn}
Cushman R., {\'S}niatycki J., Dif\/ferential structure of orbit spaces,
 \href{http://dx.doi.org/10.4153/CJM-2001-029-1}{\textit{Canad.~J. Math.}} \textbf{53} (2001), 715--755.

\bibitem{DoIg}
Donato P., Igl{\'e}sias P., Exemples de groupes dif\/f\'eologiques: f\/lots
 irrationnels sur le tore, \textit{C.~R.~Acad. Sci. Paris S\'er.~I Math.}
 \textbf{301} (1985), 127--130.

\bibitem{DK}
Duistermaat J.J., Kolk J.A.C., Lie groups, \href{http://dx.doi.org/10.1007/978-3-642-56936-4}{\textit{Universitext}}, Springer-Verlag,
 Berlin, 2000.

\bibitem{GGK}
Guillemin V., Ginzburg V., Karshon Y., Moment maps, cobordisms, and
 {H}amiltonian group actions, \href{http://dx.doi.org/10.1090/surv/098}{\textit{Mathematical Surveys and Monographs}},
 Vol.~98, Amer. Math. Soc., Providence, RI, 2002.

\bibitem{haefliger}
Haef\/liger A., Groupo\"\i des d'holonomie et classif\/iants, \textit{Ast\'erisque}
 \textbf{116} (1984), 70--97.

\bibitem{hochschild}
Hochschild G., The structure of {L}ie groups, Holden-Day, Inc., San
 Francisco~-- London~-- Amsterdam, 1965.

\bibitem{IZKZ}
Iglesias P., Karshon Y., Zadka M., Orbifolds as dif\/feologies, \href{http://dx.doi.org/10.1090/S0002-9947-10-05006-3}{\textit{Trans.
 Amer. Math. Soc.}} \textbf{362} (2010), 2811--2831, \href{http://arxiv.org/abs/math.DG/0501093}{math.DG/0501093}.

\bibitem{iglesias}
Iglesias-Zemmour P., Dif\/feology, \href{http://dx.doi.org/10.1090/surv/185}{\textit{Mathematical Surveys and Monographs}},
 Vol.~185, Amer. Math. Soc., Providence, RI, 2013.

\bibitem{IK}
Iglesias-Zemmour P., Karshon Y., Smooth {L}ie group actions are parametrized
 dif\/feological subgroups, \href{http://dx.doi.org/10.1090/S0002-9939-2011-11301-7}{\textit{Proc. Amer. Math. Soc.}} \textbf{140} (2012),
 731--739, \href{http://arxiv.org/abs/1012.0107}{arXiv:1012.0107}.

\bibitem{zoghi}
Karshon Y., Zoghi M., Orbifold groupoids and their underlying dif\/feology,
 {E}arlier version posted at
 \url{http://www.math.toronto.edu/mzoghi/research/Groupoids.pdf} and
 summarized in Zoghi's PhD thesis, University of Toronto, 2010.

\bibitem{koszul}
Koszul J.L., Sur certains groupes de transformations de {L}ie, in G\'eom\'etrie
 dif\/f\'erentielle ({C}olloques {I}nternationaux du {C}entre {N}ational de la
 {R}echerche {S}cientif\/ique, {S}trasbourg, 1953), Centre National de la
 Recherche Scientif\/ique, Paris, 1953, 137--141.

\bibitem{LMTW}
Lerman E., Meinrenken E., Tolman S., Woodward C., Nonabelian convexity by
 symplectic cuts, \href{http://dx.doi.org/10.1016/S0040-9383(97)00030-X}{\textit{Topology}} \textbf{37} (1998), 245--259,
 \href{http://arxiv.org/abs/dg-ga/9603015}{dg-ga/9603015}.

\bibitem{palais}
Palais R.S., On the existence of slices for actions of non-compact {L}ie
 groups, \href{http://dx.doi.org/10.2307/1970335}{\textit{Ann. of Math.}} \textbf{73} (1961), 295--323.

\bibitem{satake56}
Satake I., On a generalization of the notion of manifold, \href{http://dx.doi.org/10.1073/pnas.42.6.359}{\textit{Proc. Nat.
 Acad. Sci. USA}} \textbf{42} (1956), 359--363.

\bibitem{satake57}
Satake I., The {G}auss--{B}onnet theorem for {$V$}-manifolds, \href{http://doi.org/10.2969/jmsj/00940464}{\textit{J.~Math.
 Soc. Japan}} \textbf{9} (1957), 464--492.

\bibitem{schwarz}
Schwarz G.W., Smooth functions invariant under the action of a compact {L}ie
 group, \href{http://doi.org/10.1016/0040-9383(75)90036-1}{\textit{Topology}} \textbf{14} (1975), 63--68.

\bibitem{sjamaar}
Sjamaar R., A de {R}ham theorem for symplectic quotients, \href{http://dx.doi.org/10.2140/pjm.2005.220.153}{\textit{Pacific~J.
 Math.}} \textbf{220} (2005), 153--166, \href{http://arxiv.org/abs/math.SG/0208080}{math.SG/0208080}.

\bibitem{SL}
Sjamaar R., Lerman E., Stratif\/ied symplectic spaces and reduction, \href{http://dx.doi.org/10.2307/2944350}{\textit{Ann.
 of Math.}} \textbf{134} (1991), 375--422.

\bibitem{sniatycki}
{\'S}niatycki J., Dif\/ferential geometry of singular spaces and reduction of
 symmetry, \href{http://dx.doi.org/10.1017/CBO9781139136990}{\textit{New Mathematical Monographs}}, Vol.~23, Cambridge University
 Press, Cambridge, 2013.

\bibitem{souriau}
Souriau J.-M., Groupes dif\/f\'erentiels, in Dif\/ferential Geometrical Methods in
 Mathematical Physics ({P}roc. {C}onf., {A}ix-en-{P}rovence/{S}alamanca,
 1979), \href{http://dx.doi.org/10.1007/BFb0089728}{\textit{Lecture Notes in Math.}}, Vol.~836, Springer, Berlin~-- New
 York, 1980, 91--128.

\bibitem{watts-masters}
Watts J., The calculus on subcartesian spaces, {M.Sc.}~Thesis, University of
 Calgary, Canada, 2006.

\bibitem{watts}
Watts J., Dif\/feologies, dif\/ferential spaces, and symplectic geometry, Ph.D.~Thesis, {U}niversity of Toronto, Canada, 2012.

\bibitem{watts-groupoids}
Watts J., The orbit space and basic forms of a proper Lie groupoid,
 \href{http://arxiv.org/abs/1309.3001}{arXiv:1309.3001}.

\bibitem{WW}
Watts J., Wolbert S., Dif\/feology: a concrete foundation for stacks,
 \href{http://arxiv.org/abs/1406.1392}{arXiv:1406.1392}.

\end{thebibliography}
\end{document}